\documentclass[11pt,reqno]{amsart}

\usepackage[T1]{fontenc}

\usepackage{amsmath,xypic,tikz-cd}								
\usepackage{amssymb}
\usepackage{amsthm}
\usepackage{amscd}
\usepackage{amsfonts}
\usepackage{stmaryrd}
\usepackage[all]{xy}

\usepackage{euler}

\usepackage{extarrows}

\usepackage[colorlinks, linktocpage, citecolor = blue, linkcolor = blue]{hyperref}
\usepackage{color}

\usepackage{tikz}									
\usetikzlibrary{matrix}
\usetikzlibrary{patterns}
\usetikzlibrary{matrix}
\usetikzlibrary{positioning}
\usetikzlibrary{decorations.pathmorphing}

\usepackage{fullpage}
\usepackage{enumerate}

\newtheorem{theorem}{Theorem}[section]
\newtheorem{lemma}[theorem]{Lemma}

\newtheorem{corollary}[theorem]{Corollary}

\theoremstyle{definition}
\newtheorem{definition}[theorem]{Definition}
\newtheorem{example}[theorem]{Example}

\newtheorem{remark}[theorem]{Remark}

\newcommand{\calM}{\mathcal{M}}

\newcommand{\calP}{\mathcal{P}}
\newcommand{\calQ}{\mathcal{Q}}

\newcommand{\Moduli}{\overline{\mathcal{M}}}
\newcommand{\UF}{\overline{\mathcal{U}}}

\newcommand{\decorate}{\frac{\psi_{\bullet_j}^{\alpha_j}}{(-\psi_{\bullet_j}-\psi_{\star_j})^{1-\delta_j}}}
\newcommand{\decorateone}{\frac{\psi_{\bullet_1}^{\alpha_1}}{(-\psi_{\bullet_1}-\psi_{\star_1})^{1-\delta_1}}}
\newcommand{\decorateoneb}{\frac{\psi_{\bullet_1}^{\alpha_1}}{(-\psi_{\bullet_1}-\psi_{\star_1})}}

\title[]{Intersections of $\omega$ classes in $\Moduli_{g,n}$}

\author{Vance Blankers}
\address{Department of Mathematics, Colorado State University, Fort Collins, Colorado 80523-1874}
\email{\href{mailto:blankers@mail.colostate.edu }{blankers@mail.colostate.edu }}
\thanks{V.B. was supported by NSF FRG grant 1159964 (PI: Renzo Cavalieri)}

\author{Renzo Cavalieri}
\address{Department of Mathematics, Colorado State University, Fort Collins, Colorado 80523-1874}
\email{\href{mailto:renzo@math.colostate.edu}{renzo@math.colostate.edu}}
\thanks{R.C. acknowledges support from Simons Foundation Collaboration Grant 420720}


\begin{document}

\begin{abstract} 
We provide a graph formula which describes an arbitrary monomial in $\omega$  classes (also referred to as {\it stable} $\psi$ classes) in terms of a simple family of dual graphs ({\it pinwheel graphs}) with edges decorated by rational functions in $\psi$ classes. We deduce some numerical consequences and in particular a  combinatorial formula expressing top intersections of $\kappa$ classes on $\Moduli_g$ in terms of top intersections of $\psi$ classes. 
\end{abstract}
\maketitle

\setcounter{tocdepth}{1}





\section*{Introduction}

In \cite{m:taegotmsoc}, Mumford initiated the study of {\it tautological} intersection theory of the moduli spaces of curves. Realizing that the full Chow ring of $\Moduli_{g,n}$ was largely out of reach, he sought to identify a set of classes that mediates between the following tension: on the one hand, giving a theory with manageable algebraic structure, on the other capturing a large number of Chow classes that are geometrically defined.  Remarkably, the absolute minimal requirements, i.e. to be subalgebras which  contain fundamental classes, and to be closed under push-forwards via the natural gluing and forgetful morphisms, already produce a quite robust theory: many interesting classes such as the Chern classes of the Hodge bundle, Hurwitz, Gromov-Witten and Brill-Noether classes are { tautological} (\cite{fp:rmatc}). 

In \cite{gp:tg}, Graber and Pandharipande exhibit a set of additive generators for the tautological ring, parametrized by dual graphs with vertices decorated by monomials in $\kappa$ classes  and flags decorated by powers of $\psi$ classes. Following this result, a natural direction of investigation is to describe geometrically defined tautological classes in terms of these standard generators.
There are two technical obstructions to this plan, first that the ranks of the graded parts of the tautological rings for positive dimensional classes grow rather fast as $g$ or $n$ get larger, and second there  are many relations among the standard generators and generally no canonical or especially meaningful choice for a basis of the tautological ring. For this reason until recently the study of the structure of families of tautological cycles was mostly restricted to top intersections (e.g. Gromov-Witten invariants, Hurwitz numbers); these are cycles of dimension $0$, hence proportional to the class of a point. When positive dimensional tautological classes come in families, it is desirable to describe them in a way that highlights such structure. A {\it graph formula}, i.e. a formula that describes a tautological class as a sum over graphs with local (vertex, edge, flag) decorations is a combinatorially pleasing and effective way to achieve this goal.
In \cite{proofofpixton}, the authors prove a remarkable graph formula, conjectured by Pixton, for double ramification loci. In \cite{cat}, the second author and Tarasca give graph formulas for the classes of genus $2$ curves with marked Weierstrass points. 
 
The tautological $\psi$ classes (Definition \ref{def-psi}) play a central role in this theory. They arise when performing non-transversal intersections of boundary strata, which causes them to appear as edge decorations for Pixton's graph formulas. Tautological $\psi$ classes are almost stable under pull-back: Lemma \ref{lem-comp} describes the rational tail correction between the pull-back of a $\psi$ class from $\Moduli_{g,n}$ and the  corresponding $\psi$ class on $\Moduli_{g, n+1}$.
When imposing geometric conditions at marked points (for example in \cite{cat} the property of being a Weierstrass point), it is often more convenient to work with the stable version of $\psi$ classes, which are called $\omega$ classes (Definition \ref{def-omega}). 

It is simple to describe the difference between a $\psi_i$ class and the corresponding $\omega_i$ class: it consists of all divisors where the $i$-th mark is contained in a rational tail. Tracking down the corrections when intersecting a certain number of these classes quickly becomes unwieldy, because one is performing both transverse and non transverse intersection of boundary divisors whose irreducible components grow exponentially with the number of marked points.

The main result of this article is Theorem \ref{omega}, which gives a graph formula for an arbitrary monomial in $\omega$ classes in terms of a family of graphs of rational tails, with edges decorated by simple rational functions in $\psi$ classes. The simplicity of the formula, which certainly surpassed our initial expectations, witnesses the high degree of symmetry and combinatorial structure present in the problem. We hope and expect that formula \eqref{omegafor} will be useful for further development and manipulation of other graph formulas in the tautological ring, thus giving a positive contribution to the successful development of a {\it calculus for the moduli space of curves}, as advocated by Pandharipande in his 2015 AMS Algebraic Geometry Program plenary lecture (\cite{p:cmsoc}).
\vspace{0.2cm}

While we expect the readers to have had some exposure to the moduli space of curves, we are hoping to be able to communicate and attract the interest of combinatorially inclined algebraic geometers whom might not be experts in $\Moduli_{g,n}$. For this reason, Section \ref{sec-prel} provides a brief introduction to  basic facts and techniques about tautological classes, together with skeletal proofs for some basic computations on $\psi$ and $\omega$ classes which are well known to the experts, but hard to explicitly find in the literature.
The reader interested in more comprehensive references may look at \cite{hm:moc} for a general introduction to the theory, at \cite{y:inom} for a working reference to algorithms for intersecting boundary strata, and  at the unpublished notes \cite{k:pc} for psi classes.
Section \ref{sec-mt} states and proves Theorem \ref{omega}: the proof is a combination of induction and manipulation of rational tail boundary strata to recognize genus zero $\psi$ classes contributing to the formula.
In Section \ref{sec-numint}, we specialize to the case of top intersections of $\omega$ classes: projection formula gives a direct and simple relation with certain top intersections of $\kappa$ classes on $\Moduli_g$, and thus \eqref{omegafor} produces a simple and explicit formula that relates top intersections of $\kappa$ and $\psi$ classes (though we remark that the fact that the two theories are equivalent was well known).

\subsection*{Acknowledgements} The authors would like to thank Emily Clader, Dusty Ross and Mark Shoemaker  for interesting conversations related to this work. The second author wishes to thank the Fields Institute of Mathematics: this project began during his visit to the special program in Combinatorial Algebraic Geometry in Fall 2016. In particular the collaboration with Nicola Tarasca was instrumental both for gaining experience and manuality with graph formulas, and realizing the need for an efficient relation between $\psi$ classes and their stable version.

\section{Preliminaries} \label{sec-prel}

Given two non-negative integers $g,n$ satisfying $2g-2+n >0$, we denote by $\Moduli_{g,n}$ the fine moduli space for families of Deligne-Mumford stable curves of genus $g$ with $n$ marked points. The space $\Moduli_{g,n}$ is a smooth and projective DM stack of dimension $3g-3+n$, and it is stratified by locally closed substacks parameterizing topologically equivalent pointed curves; strata are naturally indexed by {\bf dual graphs}, constructed as follows. Given a pointed curve $(C, p_1, \ldots, p_n)$,  consider the normalization $\nu:{C}'\to C$ of the curve $C$; attach a flag to each point $\nu^{-1}(p_i)$, labeled by the corresponding mark; for each node $x\in C$, connect by an  edge the two points in $\nu^{-1}(x)$; then contract each irreducible component of the normalization to a vertex, and label it by the genus of the component.
An example of the graph resulting from this construction is illustrated in Figure \ref{fig-dual}. 

Since we are interested in intersecting closed cycles, we denote by $\Delta_\Gamma$ the closure of the stratum identified by a dual graph $\Gamma$. We also adopt the common abuse of calling {\bf boundary stratum} the cycle obtained as the closure of a stratum.  The codimension of $\Delta_\Gamma$ in $\Moduli_{g,n}$ is equal to the number of edges of $\Gamma$; further, we have that $\Delta_{\Gamma_1} \subseteq \Delta_{\Gamma_2}$ if and only if $\Gamma_2$ is obtained from $\Gamma_1$ by a sequence of edge contractions. An {\bf edge contraction} is an operation on dual graphs that consists in contracting an edge and  either adding the genera of the two distinct vertices that get identified in the process, or, if the edge is a loop, adding one to the genus of the vertex to which it is attached. 
 \begin{figure}[tb]
\begin{tikzpicture}

\draw[very thick] (0,0) ellipse (4cm and 1cm);

\draw[very thick] (1,0) .. controls (1.5,-0.5) and (2.5,-0.5) .. (3,0);
\draw[very thick] (1.2,-0.18) .. controls (1.5,0.3) and (2.5,0.3) .. (2.8,-0.18);

\draw[very thick] (-1,0) .. controls (-1.5,-0.5) and (-2.5,-0.5) .. (-3,0);
\draw[very thick] (-1.2,-0.18) .. controls (-1.5,0.3) and (-2.5,0.3) .. (-2.8,-0.18);

\fill (0,0) circle (0.10);
\node at (0.3,0.2) {$5$};

\draw[very thick]  (2.5,1.8) circle (1);
\fill (2,1.5) circle (0.10);
\node at (2.3,1.7) {$3$};
\fill (3,2) circle (0.10);
\node at (2.7,2.2) {$4$};

\draw[very thick]  (-2.5,1.8) circle (1);
\fill (-2,1.5) circle (0.10);
\node at (-2.3,1.7) {$2$};
\fill (-3,2) circle (0.10);
\node at (-2.7,2.2) {$1$};


\draw[very thick]  (8,0) circle (0.30);
\node at (8,0) {$2$};
\draw[very thick] (8.2,0.2) -- (10,2);
\draw[very thick] (7.8,0.2) -- (6,2);
\fill (9.2,1.2) circle (0.15);
\fill (6.8,1.2) circle (0.15);

\draw[very thick] (9.2,1.2) -- (8.4,2);
\draw[very thick] (6.8,1.2) -- (7.6,2);

\draw[very thick] (8,-0.3) -- (8,- 1);

\node at (10.1,2.2) {$4$};
\node at (8.3,2.2) {$3$};
\node at (5.9,2.2) {$1$};
\node at (7.7,2.2) {$2$};
\node at (8, -1.2) {$5$};
\end{tikzpicture}
\caption{On the left-hand side, a topological type of genus $2$ curves with  $5$ marked poinsts; on the right-hand side the corresponding dual graph. Vertices corresponding to genus zero components are represented by filled in dots.}
\label{fig-dual}
\end{figure}
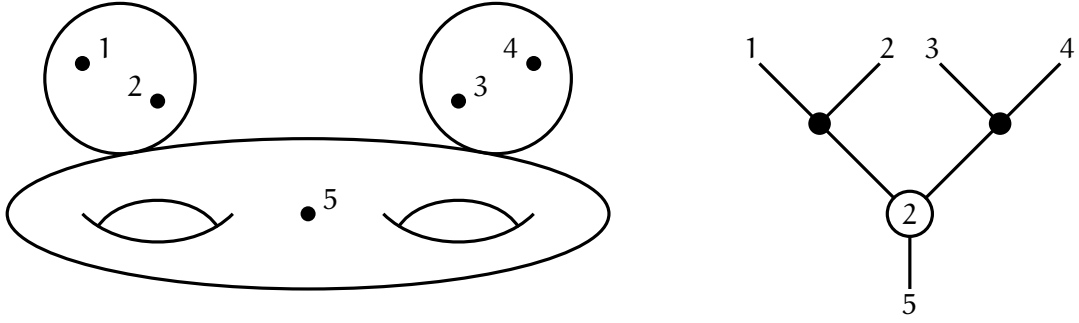

Strata arise as images of tautological {\bf gluing morphisms}. Given a dual graph $\Gamma$, we define
\begin{equation}
gl_\Gamma: \prod_{v\in V(\Gamma)} \Moduli_{g(v), val(v)} \to \Moduli_{g,n}
\end{equation}  
to be the morphism that glues marked points corresponding to pairs of flags that form an edge of the dual graph. The morphism $gl(\Gamma)$ is finite  onto $\Delta_\Gamma$ of degree $|Aut(\Gamma)|$. 

Given a mark $i \in [n+1]$, there is a {\bf forgetful morphism}
\begin{equation}
\pi_i: \Moduli_{g,n+1} \to \Moduli_{g, [n+1]\smallsetminus \{i\}} \cong \Moduli_{g, n}
\end{equation}
which assigns to an $(n+1)$-pointed curve $(C,p_1, \ldots, p_{n+1})$ the $n$-pointed curve obtained by forgetting the $i$-th marked point and  contracting any rational component of $C$ which has less than three special points (marks or nodes). The morphism $\pi_i$ functions as a universal family  for $\Moduli_{g,n}$, and so in particular the universal curve  $\UF_{g,n} \to \Moduli_{g,n}$ may be identified with $\Moduli_{g,n+1}$.

The {\bf $i$-th tautological section} 
\begin{equation}
\sigma_i: \Moduli_{g,n} \to \UF_{g,n} \cong \Moduli_{g,n+1}
\end{equation}
assigns to an $n$-pointed curve $(C,p_1, \ldots, p_n)$ the point $p_i$ in the fiber over $(C,p_1, \ldots, p_n)$ in the universal curve. Such points corresponds to the $(n+1)-$th pointed curve obtained by attaching a rational component to the point $p_i\in C$ and placing arbitrarily the marks $p_i$ and $p_{n+1}$ on the new rational component. Via the identification of the universal map with a forgetful morphism, the section $\sigma_i$ may be viewed as a gluing morphism and its image as a boundary stratum, as illustrated in the following diagram:
\begin{equation}
\xymatrix{
\UF_{g,n} \ar[rr]^{\cong} & & \Moduli_{g, n+1}\\
Im(\sigma_i) \ar@{}[u]|-*[@]{\subseteq}& & D_{i,n+1}\ar@{}[u]|-*[@]{\subseteq}\\
\Moduli_{g,n} \ar[rr]^{\hspace{-1cm} \cong}  \ar[u]^{\sigma_i}&  & \Moduli_{g, [n]\smallsetminus \{i\} \cup \{\bullet\}} \times \Moduli_{0, \{\star, i, n+1\}} \ar[u]^{gl_{D}}
}
\end{equation}

We consider all $\Moduli_{g,n}$'s (for all possible values of $g,n$) as a system of moduli spaces connected by the tautological morphisms, and define the {\bf tautological ring} $\mathcal{R} = \{R^\ast(\Moduli_{g,n})\}_{g,n}$ of this system to be the smallest system of subrings of the Chow ring of  each $\Moduli_{g,n}$, containing all fundamental classes $[\Moduli_{g,n}]$ and closed under push-forwards and pull-backs via the tautological morphisms. Clearly boundary strata are elements of the tautological ring. When studying non-transverse intersections of boundary strata (which are tautological by definition), a new family of interesting tautological classes are introduced, which we now describe.

\begin{definition}\label{def-psi}
For any choice of mark $i \in [n]$, the class $\psi_i \in R^1(\Moduli_{g,n})$ is defined in one of the following equivalent ways:
\begin{enumerate}
	\item $\psi_i = c_1(\mathbb{L}_i)$, where the $i$-th cotangent line bundle $\mathbb{L}_i$ is defined by a natural identification of its fiber over a point $(C, p_1, \ldots, p_n)$ with the cotangent space $T^\ast_{p_i}(C)$.
	\item $\psi_i = \sigma_i^\ast(\omega_{\pi})$, where $\omega_{\pi}$ denotes the relative dualizing sheaf of the universal family $\pi: \UF_{g,n} \to \Moduli_{g,n}$.
	\item $\psi_i = - \pi_\ast(c_1(N_{\sigma_i}))$, where we denote by $N_{\sigma_i}$ the normal bundle to the image of the $i$-th tautological section, and by $\pi$ the universal family as above. 
\end{enumerate}
\end{definition}

For $g \geq 2$, the class $\psi_i$ is not equivalent to a linear combination of boundary strata. Yet the following comparison lemma tells us that most of the geometric information of these classes is contained in $\psi_1 \in R^1(\Moduli_{g,1})$.

\begin{lemma}\cite{k:pc} \label{lem-comp}
Consider the forgetful morphism $\pi_{n+1}: \Moduli_{g,n+1} \to \Moduli_{g,n}$, and let the context determine whether $\psi_i$ denotes the class on $\Moduli_{g,n}$ or $\Moduli_{g,n+1}$. For $i\in [n]$, we have:
\begin{equation} \label{pbr}
\psi_i = \pi_{n+1}^\ast (\psi_i) + D_{i,n+1},
\end{equation}
where $D_{i,n+1}$ denotes the image of the section $\sigma_i$, or equivalently the boundary divisor generically parameterizing nodal curves where one component  is rational and it hosts the $i$-th and $(n+1)$-th marks.
\end{lemma}

Equation \eqref{pbr} leads to some combinatorially appealing representations of $\psi$ classes in genus zero as sums of boundary divisors. For a  two part partition $A\cup B = [n]$ of the set of indices with  $|A|, |B| \geq 2$, we  denote by $D(A|B)$ the (class of the) boundary divisor generically parameterizing  nodal curves where the marks in the subset $A$ are in one component, and those in $B$ in the other; alternatively, the divisor isomorphic to the closure of the image of $gl_D:\Moduli_{0,A\cup\{\bullet\}}\times\Moduli_{0,B\cup\{\star\}}$ in $\Moduli_{0,n}$.

\begin{lemma}
\label{bes}
For $i,j,k$ distinct elements of the set of marks, the class $\psi_i \in R^1(\Moduli_{0,n})$ may be represented by the following expression:
\begin{equation} \label{besf}
\psi_i = \sum_{\tiny{\begin{array}{c} A\ni j,k \\ B \ni i \end{array}}} D(A|B)
\end{equation}
\end{lemma}
\begin{proof}
This result follows from iterated applications of \eqref{pbr}. The base case is $\Moduli_{0,\{i,j,k\}}\cong \Moduli_{0,3}$, where all $\psi$ classes are $0$  for dimension reasons.
\end{proof}
We note that such an expression is neither unique nor canonical, as it depends on the choice of the auxiliary marks $j,k$. As a corollary of Lemma \ref{bes} we have a canonical boundary expression for a sum of two $\psi$ classes. We give a derivation of this easy result since we could not find a reference in the literature.

\begin{lemma}\label{psisums}
Let $n\geq 3$; for any distinct $i,j\in [n]$ the following idenitity holds in the tautological ring of  $\Moduli_{0,n}$:
\begin{align}
\psi_i + \psi_j = \sum_{\tiny{\begin{array}{c} A\ni j \\ B \ni  i \end{array}}} D(A|B).
\end{align}
\end{lemma}

\begin{proof} 
Choose a point $k \in [n] - \{i,j\}$ and apply \eqref{besf} to obtain:

\begin{align} \label{psiij}
\psi_i = \sum_{\tiny{\begin{array}{c} A\ni j,k \\ B \ni  i \end{array}}} D(A|B) \ \ \ \ \ \  \psi_j = \sum_{\tiny{\begin{array}{c} A \ni j \\ B \ni i,k \end{array}}} D(A|B).
\end{align}
The lemma follows immediately from adding the two terms in \eqref{psiij}.
 \end{proof}
 
Restricting  $\psi_i$ to a boundary stratum yields the pull-back of the corresponding class $\psi_i$  from the factor hosting the $i$-th mark: in the next lemma we make this statement precise for a boundary divisor, leaving it to the reader the  generalization to an arbitrary stratum.

\begin{lemma}\label{rest}
Consider the gluing morphism
$$
gl_{D}:\Moduli_{g_1, A \cup \{\bullet\}} \times \Moduli_{g_2, B \cup \{\star\}} \to \Moduli_{g,n},
$$
whose image is the boundary divisor denoted $ D(g_1,A|g_2, B)$. Assume that $i\in A$, and denote $p_1: \Moduli_{g_1, A \cup \{\bullet\}} \times \Moduli_{g_2, B \cup \{\star\}} \to \Moduli_{g_1,A \cup \{\bullet\}}$ the first projection. Then
\begin{equation}
gl_{D}^\ast (\psi_i) = p_1^\ast (\psi_i).
\end{equation}
\end{lemma}
\begin{proof}
For a point in $D(g_1,A|g_2, B)$, i.e. a nodal pointed  curve $(C= C_1 \cup C_2, p_1, \ldots, p_n)$  (with $p_i\in C_1$), we have $T^\ast_{p_i}(C)  = T^\ast_{p_i}(C_1)$.   This implies there is an isomorphism $gl_{D}^\ast (\mathbb{L}_i) \cong p_1^\ast (\mathbb{L}_i)$, and the result follows.
\end{proof}
 
As we mentioned earlier, $\psi$ classes arise when performing non-transversal intersections of boundary strata. As before, we make a precise statement for the self intersection of a divisor, and leave the more general statement as an exercise.

 \begin{lemma} 
 With notation as in Lemma \ref{rest}, the self intersection of a boundary divisor $D(g_1,A|g_2, B) \in R^1(\Moduli_{g,n})$ is given by:
 \begin{equation}\label{nont}
 D(g_1,A|g_2, B)^2 ={gl_{D}}_\ast (-\psi_\bullet - \psi_\star ),
 \end{equation}
 where the two $\psi$ classes are understood to be pulled back via the two projections $p_1, p_2$.
 \end{lemma}
\begin{proof}
The germ of a path moving off the divisor $ D(g_1,A|g_2, B)$ consists of a first order deformation of a nodal curve $C = C_1\cup C_2$ (with appropriate sections that are disjoint from the node), with smooth generic fiber.
The local analytic expression for a one parameter smoothing of a node is $xy=t$. Here $t$, the smoothing parameter, may be identified with a local coordinate for the fiber of the normal bundle to the divisor $D(g_1,A|g_2, B)$, and $x$ and $y$ with coordinates on the tangent spaces for the components of the central fiber at the points that are glued together to give the node.
This identification yields the isomorphism
$$
N_{D(g_1,A|g_2, B)|\Moduli_{g,n}} \cong \mathbb{L}^\vee_\bullet \boxtimes  \mathbb{L}^\vee_\star,
$$
which   implies \eqref{nont}.
\end{proof} 
 
We now define $\omega$ classes, sometimes also called {\it stable $\psi$ classes}, which are just pull-backs of $\psi$ classes from spaces of curves with only one mark.
 
\begin{definition}\label{def-omega}
Let $g,n\geq 1$, $i\in [n]$, and let $\rho_i:\Moduli_{g,n}\to\Moduli_{g,\{i\}}$ be the {\it rememberful} morphism, i.e. a composition of forgetful morphisms for  all but the $i$-th mark. Then we define $$\omega_i := \rho_i^*\psi_i$$ in $R^1(\Moduli_{g,n}$). 
\end{definition}

Iterated applications of Lemma \ref{lem-comp} show the relation between the classes $\psi_i$ and $\omega_i$ on $\Moduli_{g,n}$. Denote by $D(A|B)$ the divisor $D(g,A|0,B)$. We call any boundary stratum where all the genus is concentrated at one vertex of the dual graph a stratum of {\bf rational tails} type.

\begin{lemma}
Let $g,n\geq 1$ and $i\in [n]$. Then:
\begin{equation} \label{ompsirel}
\psi_i= \omega_i + \sum_{B\ni i} D(A|B)
\end{equation}
\end{lemma}
In words, this means that $\psi_i$ is obtained from $\omega_i$ by adding all divisors of rational tails where the $i$-th mark is contained in the rational component.

We conclude this section by discussing how $\omega$ classes restrict to boundary strata. When the $i$-th mark is on a component that remains stable after forgetting all other marks, then one can show, with an identical proof to the case of $\psi$ classes, that  $\omega_i$ restricts to the class $\omega_i$ pulled back via the projection from the factor containing the $i$-th mark. Things are more interesting when the $i$-th mark is on a rational tail, as we show in the next lemma.

\begin{lemma}\label{falldown}
Let $D(A|B)$ be a divisor of rational tails, and suppose the $i$-th marked point is on the rational component ($i\in B$). Then 
\begin{equation}\omega_i \cdot D(A|B) = {gl_{D}}_\ast ( \omega_{\bullet}),
\end{equation}
where as usual $\omega_{\bullet}$ denotes the class pulled back from the projection $p_1: \Moduli_{g, A \cup \{\bullet\}} \times \Moduli_{0, B \cup \{\star\}} \to \Moduli_{g,A \cup \{\bullet\}}$
\end{lemma}

\begin{proof}  Consider the following diagram:

\begin{center}
\begin{tikzcd}[row sep=scriptsize, column sep=scriptsize]
\Moduli_{g,A\cup\{\bullet\}}\times \Moduli_{0,B\cup\{\star\}} \arrow[d, "gl_D"] \arrow[rd, "p_1"] \\
\Moduli_{g,n} \arrow[r,"\rho_{A\cup\{i\}}"] & \Moduli_{g,A\cup\{i\}} \arrow[r, "\rho_i"] & \Moduli_{g,\{i\}}
\end{tikzcd}
\end{center}

The map $p_1$ is the projection onto $\Moduli_{g,A\cup\{\bullet\}}$ composed with the isomorphism which relabels $\bullet$ to $i$. Then by the commutativity of the diagram and the definition of $\omega_i$,
\begin{align*}
\omega_i \cdot D(A|B)&= {gl_{D}}_\ast gl_D^*(\omega_i )\\
&= {gl_{D}}_\ast gl_D^*\rho_{A\cup\{i\}}^*\rho_i^*(\psi_i) \\
&= {gl_{D}}_\ast p_1^*\rho_i^*(\psi_i) \\
&= {gl_{D}}_\ast(\omega_{\bullet}).
\end{align*}
\end{proof}

\begin{remark}
In order to streamline notation, when  we write $\psi$ and $\omega$ classes relative to some  flag of a dual graph, we implicitly mean the push-forward via the appropriate gluing morphism of the pull-back via the projection to the factor hosting the flag, of the corresponding class. See Figure \ref{fig-not} for an illustration that should make this tongue twister much more clear. 
\end{remark}

\begin{figure}[bt]

\begin{tikzpicture}
\draw[very thick]  (0,0) circle (0.30);
\node at (0,0) {$g$};
\draw[very thick] (0.30,0) -- (1.5,0);
\fill (1.5,0) circle (0.10);
\draw[very thick] (1.5,0) -- (2,0);
\draw[very thick] (1.5,0) -- (2,0.5);
\draw[very thick] (1.5,0) -- (2,-0.5);
\node at (2.2,0) {$6$};
\node at (2.2,0.5) {$3$};
\node at (2.2,-0.5) {$7$};
\node at (0.55,0.25) {$\psi_\bullet$};
\node at (1.2,0.2) {$\star$};

\draw[very thick] (0,0.30) -- (0,0.7);
\node at (-0.1,0.9) {$2$};

\draw[very thick] (-0.21,0.21) -- (-0.50,0.5);
\node at (-0.7,0.6) {$5$};

\draw[very thick] (-0.30,0) -- (-.7,0);
\node at (-1,0) {$4$};

\draw[very thick] (-0.21,-0.21) -- (-.5,-.5);
\node at (-0.8,-.5) {$1$};

\end{tikzpicture}
\caption{The dual graph identifying the divisor $D(A|B)$, with $A= \{1,2,4,5\}$ and $B= \{ 3,6,7\}$. The graph is decorated with a $\psi$ class on a flag. This is shorthand for ${gl_{D}}_\ast p_1^\ast(\psi_\bullet)$.
}\label{fig-not}
\end{figure}
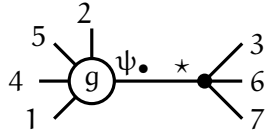

\section{Main theorem} \label{sec-mt}

In this section we state and proof the graph formula for the class of an arbitrary monomial in $\omega$ classes. We begin by introducing the family of boundary strata which appear in the formula. Throughout this section, we fix two positive integers $g$ and $n$ for genus and number of marked points.

We denote by $\calP \vdash [n]$ a partition of the set $[n]$, i.e. a collection of  pairwise disjoint subsets $P_1,\dots, P_r$ such that 
$$
P_1 \cup \ldots \cup P_r = [n].
$$
We wish to consider partitions as unordered: in other words, we identify two partitions if they differ by a permutation of the parts.
We assume all the $P_i$'s are non-empty, and say $\calP$ has length $r$ (and write $\ell(\calP)=r$).
We assign to this data a  stratum in $\overline{\calM}_{g,n}$, of codimension equal to the number of parts of $\calP$ of size greater than one. 
\begin{definition}
Given $\calP \vdash n$, when $|P_i| = 1$  denote by $\bullet_i$ the element of the singleton $P_i$.  For $|P_i|>1$, introduce new labels $\bullet_i$ and $\star_i$. The \emph{pinwheel stratum} $\Delta_{\calP}$ is the image of the gluing morphism
$$
gl_{\calP}: \Moduli_{g, \{\bullet_1, \ldots, \bullet_r\}} \times \prod_{|P_i|>1} \Moduli_{0, \{\star_i\}\cup P_i} \to \Moduli_{g,n}
$$
that glues together each $\bullet_i$ with $\star_i$.
Since  the general point of a pinwheel stratum represents a curve with no nontrivial automorphisms, the class of the stratum equals the push-forward of the fundamental class via $gl_{\calP}$.

\begin{equation}
\label{pinwheel}
[\Delta_{\calP}] = gl_{\calP \ast}([1]).
\end{equation}
\end{definition}
Figure \ref{figure_dualgraph} shows  an example of the dual graph of a generic element of a pinwheel stratum. 

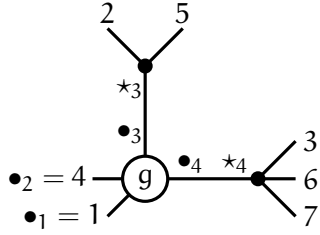
\begin{figure}[bt]

\begin{tikzpicture}

\draw[very thick]  (0,0) circle (0.30);
\node at (0,0) {$g$};

\draw[very thick] (0.30,0) -- (1.5,0);
\fill (1.5,0) circle (0.10);
\draw[very thick] (1.5,0) -- (2,0);
\draw[very thick] (1.5,0) -- (2,0.5);
\draw[very thick] (1.5,0) -- (2,-0.5);
\node at (2.2,0) {$6$};
\node at (2.2,0.5) {$3$};
\node at (2.2,-0.5) {$7$};
\node at (0.6,0.2) {$\bullet_4$};
\node at (1.2,0.2) {$\star_4$};

\draw[very thick] (0,0.30) -- (0,1.5);
\fill (0,1.5) circle (0.10);
\draw[very thick] (0,1.5) -- (-0.5,2);
\draw[very thick] (0,1.5) -- (0.5,2);
\node at (-0.5,2.2) {$2$};
\node at (0.5,2.2) {$5$};
\node at (-0.2, .6) {$\bullet_3$};
\node at (-0.2, 1.2) {$\star_3$};

\draw[very thick] (-0.30,0) -- (-.7,0);
\node at (-1.3,0) {$\bullet_2 = 4$};

\draw[very thick] (-0.21,-0.21) -- (-.5,-.5);
\node at (-1.1,-.5) {$\bullet_1 = 1$};

\end{tikzpicture}
\caption{The dual graph to the generic curve parameterized by the pinwheel stratum $\Delta_\calP$, with $\calP = \{1\},\{4\},\{2,5\},\{3,6,7\}.$ The flags of the graph are decorated with the auxiliary markings coming from the gluing morphism.}\label{figure_dualgraph}
\end{figure}

\begin{theorem}
\label{omega}
 For $1\leq i \leq n$, let $k_i$ be a non-negative integer, and let $K = \sum_{i=1}^n k_i$. 
For any partition $\calP  = \{P_1,\dots, P_r \}
\vdash [n]$, define $\alpha_j := \sum_{i\in P_j} k_i$.
With notation  as in the previous paragraph, the following formula holds in  $R^{K}(\Moduli_{g,n})$:
\begin{align}
\label{omegafor}
\prod_{i=1}^n\omega_i^{k_i} = \sum_{\calP \,\vdash [n]}[\Delta_{\calP}]\prod_{j=1}^{\ell(\calP)} \frac{\psi_{\bullet_j}^{\alpha_j}}{(-\psi_{\bullet_j}-\psi_{\star_j})^{1-\delta_j}},
\end{align}
where $\delta_j = \delta_{1,|P_j|}$ is a Kronecker delta and
we follow the standard convention of considering negative powers of $\psi$  equal to $0$.
\end{theorem}
\begin{remark}\label{rem:intfor}
In formula \eqref{omegafor}, the denominator of  the rational function is intended to be expanded as a geometric series in $\psi_\star/\psi_\bullet$. If $|P_j|>1$, 	we have
\begin{equation}\label{ratfun}
 \frac{\psi_{\bullet_j}^{\alpha_j}}{(-\psi_{\bullet_j}-\psi_{\star_j})} = - \psi_{\bullet_j}^{\alpha_j-1}+ \psi_{\bullet_j}^{\alpha_j-2} \psi_{\star_j} - \psi_{\bullet_j}^{\alpha_j-3} \psi_{\star_j}^{2}+ \ldots 
\end{equation}
The sum in \eqref{ratfun} is finite since we defined negative powers of $\psi$ to vanish. We also observe that if $\alpha_j=0$, the right-hand side of \eqref{ratfun}  equals $0$. Hence  the formula is supported on pinwheel strata where each rational tail has at least one point $i$ with strictly positive $k_i$.
\end{remark}

Before we start the proof of Theorem \ref{omega}, here is a small example of the statement.
\begin{example}\label{ex-for}
On $\Moduli_{g,3}$, we have:
\begin{equation} \omega_1^3 \omega_2^2 = \psi_1^3 \psi_2^2 
- \psi_\bullet^4[\Delta_{\{1,2\}\{3\}}] 
- \psi_\bullet^4 \psi_2^2 [\Delta_{\{1,3\}\{2\}}]
- \psi_1^3 \psi_\bullet^2[\Delta_{\{1\}\{2,3\}}]
- (\psi_\bullet^4- \psi_\bullet^3 \psi_\star) [\Delta_{\{1,2, 3\}}]
\end{equation}
The dual graphs to the strata are illustrated in Figure \ref{dg}.
\end{example}

\begin{figure}[bt]

\begin{tikzpicture}

\draw[very thick]  (0,0) circle (0.30);
\node at (0,0) {$g$};
\draw[very thick] (0.30,0) -- (1.5,0);
\fill (1.5,0) circle (0.10);
\draw[very thick] (1.5,0) -- (2,0.5);
\draw[very thick] (1.5,0) -- (2,-0.5);
\node at (2.2,0.5) {$1$};
\node at (2.2,-0.5) {$2$};
\node at (0.6,0.2) {$\bullet$};
\node at (1.2,0.2) {$\star$};
\draw[very thick] (-0.30,0) -- (-.7,0);
\node at (-0.9,0) {$ 3$};
\node at (1,-1){\Large{$\Gamma_{\Delta_{\{1,2\},\{3\}}}$}};
\draw[very thick]  (4,0) circle (0.30);
\node at (4,0) {$g$};
\draw[very thick] (4.30,0) -- (5.5,0);
\fill (5.5,0) circle (0.10);
\draw[very thick] (5.5,0) -- (6,0.5);
\draw[very thick] (5.5,0) -- (6,-0.5);
\node at (6.2,0.5) {$1$};
\node at (6.2,-0.5) {$3$};
\node at (4.6,0.2) {$\bullet$};
\node at (5.2,0.2) {$\star$};
\draw[very thick] (3.70,0) -- (3.3,0);
\node at (3.1,0) {$ 2$};
\node at (5,-1){\Large{$\Gamma_{\Delta_{\{1,3\},\{2\}}}$}};
\draw[very thick]  (8,0) circle (0.30);
\node at (8,0) {$g$};
\draw[very thick] (8.30,0) -- (9.5,0);
\fill (9.5,0) circle (0.10);
\draw[very thick] (9.5,0) -- (10,0.5);
\draw[very thick] (9.5,0) -- (10,-0.5);
\node at (10.2,0.5) {$2$};
\node at (10.2,-0.5) {$3$};
\node at (8.6,0.2) {$\bullet$};
\node at (9.2,0.2) {$\star$};
\draw[very thick] (7.70,0) -- (7.3,0);
\node at (7.1,0) {$ 1$};
\node at (9,-1){\Large{$\Gamma_{\Delta_{\{1\},\{2,3\}}}$}};
\draw[very thick]  (12,0) circle (0.30);
\node at (12,0) {$g$};
\draw[very thick] (12.30,0) -- (13.5,0);
\fill (13.5,0) circle (0.10);
\draw[very thick] (13.5,0) -- (14,0);
\draw[very thick] (13.5,0) -- (14,0.5);
\draw[very thick] (13.5,0) -- (14,-0.5);
\node at (14.2,0) {$2$};
\node at (14.2,0.5) {$1$};
\node at (14.2,-0.5) {$3$};
\node at (12.6,0.2) {$\bullet$};
\node at (13.2,0.2) {$\star$};
\node at (13,-1){\Large{$\Gamma_{\Delta_{\{1,2,3\}}}$}};
\end{tikzpicture}
\caption{The dual graphs for the strata in Example \ref{ex-for}.}\label{dg}
\end{figure}
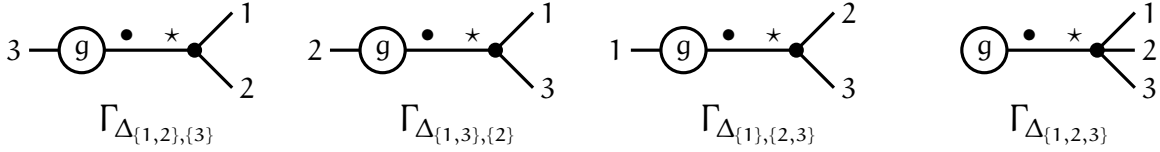

 \begin{proof}
 We proceed via induction on $n$ and the total power $K =\sum_{i=1}^n k_i$. We prove the base case $K=1$ for every $n$; without loss of generality let $k_1 = 1$ and  $k_{i} = 0$ for $i\not=1$. On the left-hand side of \eqref{omegafor}, we just have $\omega_1$.
On the right-hand side, the  pinwheel stratum  corresponding to the partition $\calP = \{\{1\},\{2\},\dots,\{n\}\}$ is $\Moduli_{g,n}$; it appears in \eqref{omegafor} with coefficient  $\psi_1$. 
As seen in Remark \ref{rem:intfor}, non-zero contributions only come from strata where each part of size greater than one has a point with non-zero $k_i$. In this case, this only leaves partitions with exactly one part $P_1$ of size greater than one, and further it must be that $1\in P_1$. We have $[\Delta_\calP] = D([n]\smallsetminus P_1|P_1)$, $\alpha_1 = 1$, and all other $\alpha$'s equal zero.  The coefficient from the first part is
 $$\frac{\psi_{\bullet_j}}{-\psi_{\bullet_j}-\psi_{\star_j}} = -1 + \frac{\psi_{\star_j}}{\psi_{\bullet_j}} - \cdots = -1.$$ 
 All other parts are singletons with $\alpha=0$, and hence  each contributes $\frac{\psi_{\bullet}^0}{1} = 1$ to the product. 
Thus equation \eqref{omegafor} becomes 
\begin{align}
\omega_1 
&= \psi_1 - \sum_{1\in B} D(A|B),
\end{align}
which we have seen in \eqref{ompsirel}. The base case is established.\\

Assume \eqref{omegafor} holds for total monomial power $K \leq m$ for some $m\in\mathbb{Z}$ and for all spaces with fewer than $n$ marked points. We hold $n$ fixed and increase  $K$ by 1 by multiplying, again without loss of generality, by $\omega_1$. We have
\begin{align} \label{for:ind}
\left(\prod_{i=1}^n \omega_i^{k_i}\right)\cdot\omega_1 &=  \prod_{i=1}^n \omega_i^{k_i} \left(\psi_1 - \sum_{1\in B} D(A|B) \right) \nonumber \\
&=\left(\prod_{i=1}^n \omega_i^{k_i}\right)\cdot\psi_1 - \sum_{1\in B} \left(\prod_{i=1}^n \omega_i^{k_i}\right)  D(A|B). 
\end{align}
We may  assume $1\in P_1$. We examine each of the summands on the right-hand side of \eqref{for:ind}. For the first term, by inductive hypothesis we have
\begin{align}
\left(\prod_{i=1}^n \omega_i^{k_i}\right)\cdot\psi_1 &= \left(\sum_{\calP\,\vdash [n]}[\Delta_{\calP}] \prod_{j=1}^{\ell(\calP)} \decorate \right)\cdot \psi_1 \nonumber \\
&= \sum_{|P_1|=1} [\Delta_{\calP}] \prod_{j=2}^{\ell(\calP)} \decorate \cdot \psi_1^{k_1+1} \nonumber \\
&\hspace{1cm} + \sum_{|P_1|>1} \left([\Delta_{\calP}]\cdot \psi_1\right) \prod_{j=1}^{\ell(\calP)} \decorate. \label{eq:lhs}
\end{align}
Note  that for the $|P_1|=1$ cases, the denominator for the $j = 1$ term is 1, as $\delta_1 = \delta_{1,|P_1|} = 1$, and $\{\bullet_1\} = \{1\}$ by the  convention adopted in defining $[\Delta_{\calP}]$.

We now turn to the second summand in \eqref{for:ind}. We rename $B = P_1$ to emphasize that  the point $1$  belongs to this subset (now $A = [n]\backslash P_1$) 
and note that  summing over all divisors $D(A|B)$ with $1\in B$  is equivalent to summing over all 
$|P_1|>1$. We denote $gl_{P_1}: \Moduli_{g, A\cup \{L_1\}} \times \Moduli_{0,P_1\cup\{R_1\}} \to \Moduli_{g,n}$ the gluing morphism whose image is $D(A|P_1)$.
Then we have:
\begin{align}
\hspace{-1cm}\sum_{1\in B} \left(\prod_{i=1}^n \omega_i^{k_i}\right)  & D(A|B) = \sum_{|P_1|>1} \left(\prod_{i=1}^n \omega_i^{k_i}\right)  D(A|P_1) \nonumber \\
&\stackrel{Lemma \ref{falldown}}{=} \sum_{|P_1|>1} {gl_{P_1}}_*\left(\prod_{i\in A}\omega_i^{k_i}\cdot \omega_{L_1}^{\sum_{i\in P_1} k_i} \right)  \nonumber\\
&= \sum_{|P_1|>1} {gl_{P_1}}_*\left(\sum_{\calQ\,\vdash A\cup\{L_1\}} [\Delta_{\calQ}] \prod_{j=1}^{\ell(\calQ)} \decorate  \right)  \label{eq:dc}
\end{align}

The last equality is the application of the inductive hypothesis. Adopting the convention that $L_1 \in Q_1$, we note that $\alpha_1 = \sum_{i\in Q_1\cup P_1}k_i$. We now group the partitions $\calQ$ in two groups: the first is where $Q_1$ is the singleton $\{L_1\}$: in this case ${gl_{P_1}}_\ast ([\Delta_\calQ])$ is the class of the pinwheel stratum $\Delta_\calP$, where $\calP = P_1 \cup \calQ \smallsetminus Q_1$. The second group contains all partions $\calQ$ with $|Q_1|>1$. See Figure \ref{fig-twogps} for a pictorial description. Then \eqref{eq:dc} continues:

\begin{align}
&= \sum_{|P_1|>1} [\Delta_{\calP}] \prod_{j=2}^{\ell(\calP)} \decorate \cdot \psi_{\bullet_1}^{\alpha_1}  \nonumber \\
&\hspace{1cm} + \sum_{|P_1|>1}\sum_{|Q_1|>1} {gl_{P_1}}_*\left( [\Delta_{\calQ}] \cdot \decorateone  \prod_{j=2}^{\ell(\calQ)} \decorate  \right) \nonumber \\ 
&= \sum_{|P_1|>1} [\Delta_{\calP}] \prod_{j=2}^{\ell(\calP)} \decorate \cdot \psi_{\bullet_1}^{\alpha_1} \nonumber \\
&\hspace{1cm} + \sum_{|P_1|>1} [\Delta_{\calP}] \prod_{j=1}^{\ell(\calP)} \decorate \cdot (\psi_1+\psi_{\star_1}) 
\label{eq:rhs}
\end{align}

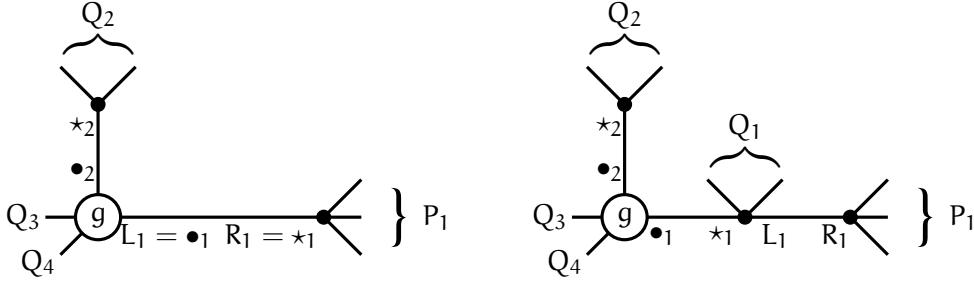
\begin{figure}[bt]

\begin{tikzpicture}
\draw[very thick]  (0,0) circle (0.30);
\node at (0,0) {$g$};

\draw[very thick] (0.30,0) -- (3,0);
\fill (3,0) circle (0.10);
\draw[very thick] (3,0) -- (3.5,0);
\draw[very thick] (3,0) -- (3.5,0.5);
\draw[very thick] (3,0) -- (3.5,-0.5);
\node at (4,0) {{\Huge{\}}}};
\node at (4.5,0) {$P_1$};
\node at (0.9,-0.25) {$L_1 = \bullet_1$};
\node at (2.3,-0.25) {$R_1= \star_1$};

\draw[very thick] (0,0.30) -- (0,1.5);
\fill (0,1.5) circle (0.10);
\draw[very thick] (0,1.5) -- (-0.5,2);
\draw[very thick] (0,1.5) -- (0.5,2);
\node at (0,2.3) {{\Large$\overbrace{}$}};
\node at (0,2.7) {$Q_2$};

\node at (-0.2, .6) {$\bullet_2$};
\node at (-0.2, 1.2) {$\star_2$};

\draw[very thick] (-0.30,0) -- (-.7,0);
\node at (-1,0) {$Q_3$};

\draw[very thick] (-0.21,-0.21) -- (-.5,-.5);
\node at (-0.8,-.6) {$Q_4$};
\draw[very thick]  (7,0) circle (0.30);
\node at (7,0) {$g$};

\draw[very thick] (7.30,0) -- (10,0);
\fill (10,0) circle (0.10);
\draw[very thick] (10,0) -- (10.5,0);
\draw[very thick] (10,0) -- (10.5,0.5);
\draw[very thick] (10,0) -- (10.5,-0.5);
\node at (11,0) {{\Huge{\}}}};
\node at (11.5,0) {$P_1$};
\node at (7.5,-0.25) {$\bullet_1$};
\node at (9,-0.25) {$L_1$};
\node at (8.3,-0.25) {$\star_1$};
\node at (9.8,-0.25) {$R_1$};

\fill (8.6,0) circle (0.10);
\draw[very thick] (8.6,0) -- (8.1,0.5);
\draw[very thick] (8.6,0) -- (9.1,0.5);
\node at (8.6,0.8) {{\Large$\overbrace{}$}};
\node at (8.6,1.2) {$Q_1$};

\draw[very thick] (7,0.30) -- (7,1.5);
\fill (7,1.5) circle (0.10);
\draw[very thick] (7,1.5) -- (6.5,2);
\draw[very thick] (7,1.5) -- (7.5,2);
\node at (7,2.3) {{\Large$\overbrace{}$}};
\node at (7,2.7) {$Q_2$};

\node at (6.8, .6) {$\bullet_2$};
\node at (6.8, 1.2) {$\star_2$};

\draw[very thick] (6.7,0) -- (6.3,0);
\node at (6,0) {$Q_3$};

\draw[very thick] (7-0.21,-0.21) -- (7-.5,-.5);
\node at (7-0.8,-.6) {$Q_4$};

\end{tikzpicture}
\caption{Examples of dual graphs corresponding to the two summands in equation \eqref{eq:rhs}. On the left-hand side we have graphs where $Q_1 = \{L_1\}$; on the right-hand side $|Q_1| >1$. }\label{fig-twogps}
\end{figure}

In the last equality we have applied Lemma \ref{psisums}, and reindexed the sum so that the new $P_1$ is equal to what used to be $P_1 \cup (Q_1 \smallsetminus L_1)$.

We  rewrite \eqref{for:ind} using \eqref{eq:lhs} and \eqref{eq:rhs};  the second term in the right-hand side of \eqref{eq:lhs}  cancels part of the  second term of the right-hand side of \eqref{eq:rhs}, and we obtain:

\begin{align}
\left(\prod_{i=1}^n \omega_i^{k_i}\right)\cdot\omega_1 = \sum_{|P_1|=1} [\Delta_{\calP}] \prod_{j=2}^{\ell(\calP)}\decorate \cdot \psi_1^{k_1+1}   - \sum_{|P_1|>1} [\Delta_{\calP}] \prod_{j=2}^{\ell(\calP)} \decorate \cdot \psi_{\bullet_1}^{\alpha_1}  \nonumber  \\
- \sum_{|P_1|>1} [\Delta_{\calP}] \prod_{j=1}^{\ell(\calP)} \decorate \cdot \psi_{\star_1} \nonumber  \\
= \sum_{|P_1|=1} [\Delta_{\calP}] \prod_{j=2}^{\ell(\calP)} \decorate \cdot \psi_1^{k_1+1}  + \sum_{|P_1|>1} [\Delta_{\calP}] \prod_{j=2}^{\ell(\calP)} \decorate \cdot \left(-\psi_{\bullet_1}^{\alpha_1}  -\psi_{\star_1}\decorateoneb \right).
\label{almostthere} 
\end{align}

We conclude the proof by observing that \eqref{almostthere} gives formula \eqref{omegafor}, with $k_1$ replaced by $k_1+1$ (and hence every occurrence of $\alpha_1$ replaced by $\alpha_1+1$): the first summand shows that the coefficients agree on the nose for partitions with $P_1 = \{1\}$; the second summand deals with partitions where $|P_1|>1$; the coefficients match after noting the elementary identity:
\begin{align*}
\frac{\psi_{\bullet_1}^{\alpha_1+1}}{-\psi_{\bullet_1}-\psi_{\star_1}} &= -\psi_{\bullet_1}^{\alpha_1} + \psi_{\bullet_1}^{\alpha_1-1}\psi_{\star_1} - \psi_{\bullet_1}^{\alpha_1-2}\psi_{\star_1}^2 + \cdots \\
&= -\psi_{\bullet_1}^{\alpha_1} - \psi_{\star_1}\cdot (-\psi_{\bullet_1}^{\alpha_1-1} + \psi_{\bullet_1}^{\alpha_1-2}\psi_{\star_1} - \cdots)
\end{align*}
\end{proof}

\section{Numerical intersections} \label{sec-numint}

In this section we observe some simple consequences of Theorem \ref{omega} for top intersections of $\omega$ classes.

\begin{theorem}\label{num}
For $1\leq i \leq n$, let $k_i$ be a non-negative integer, and let $ \sum_{i=1}^n k_i = 3g-3+n$. 
For any partition $\calP  = \{P_1,\dots, P_r \}
\vdash [n]$, define $\alpha_j := \sum_{i\in P_j} k_i$.
\begin{align}
\label{omegafor2}
\int_{\Moduli_{g,n}}\prod_{i=1}^n\omega_i^{k_i} = \sum_{\calP \,\vdash [n]}(-1)^{n+\ell(\calP)} \int_{\Moduli_{g,\ell(\calP)}}\prod_{i=1}^{\ell(\calP)} \psi_{\bullet_i}^{\alpha_i-|P_i|+1}
\end{align}
\end{theorem}

\begin{proof}
This statement follows from formula \eqref{omegafor}, by noticing the following two facts:
\begin{itemize}
	\item For any partition $\calP$, by dimension reasons the only monomial that has nonzero evaluation on $[\Delta_\calP]$ is
	\begin{equation}\label{ha}
	\prod_{|P_i|=1} \psi_{\bullet_i}^{\alpha_i}\prod_{|P_i|>1}(-1)^{|P_i|-1}\psi_{\bullet_i}^{\alpha_i-|P_i|+1}\psi_{\star_i}^{|P_i|-2 }
	\end{equation}
	\item For any $n\geq 3$, $i\in [n]$, 
	$$
	\int_{\Moduli_{0,n}} \psi_i^{n-3}  = 1,
	$$
	hence all evaluations for the classes $\psi_{\star_i}$ in \eqref{ha} contribute a factor of one to the evaluation of the monomial on $[\Delta_\calP]$.
\end{itemize}
It follows that, for every $\calP$,
$$
\int_{[\Delta_\calP]} \prod_{|P_i|=1} \psi_{\bullet_i}^{\alpha_i}\prod_{|P_i|>1}(-1)^{|P_i|-1}\psi_{\bullet_i}^{\alpha_i-|P_i|+1}\psi_{\star_i}^{|P_i|-2 }=
(-1)^{n+\ell(\calP)} \int_{\Moduli_{g,\ell(\calP)}}\prod_{i=1}^{\ell(\calP)} \psi_{\bullet_i}^{\alpha_i-|P_i|+1}
.
$$
\end{proof}
There is a natural connection between the intersection of $\omega$ classes and $\kappa$ classes. While it is well known that the intersection theory of $\kappa$ classes on $\Moduli_{g}$ is equivalent to intersections of $\psi$ classes on all $\Moduli_{g,n}$ (See, e.g. \cite[Proposition 2.3.6]{k:pc}), the relationship becomes transparent when stated in terms of $\omega$ classes.

\begin{definition}
Let $g\geq 2$, and consider $\pi_1: \Moduli_{g,1} \to \Moduli_g$. For $j\geq 1$ we define
$$
\kappa_l := {\pi_1}_\ast(\psi_1^{l+1}) \in R^l(\Moduli_{g})
$$
\end{definition}

\begin{theorem} \label{thm-ka}
Let $g\geq 2$ and let $F: \Moduli_{g,n} \to \Moduli_g$ denote the total forgetful morphism, forgetting all marks. For any collection of non-negative integers $l_i$, we have 
 \begin{align}
\label{omegaforka}
F_\ast\left(\prod_{i=1}^n\omega_i^{l_i+1}\right) = \prod_{i=1}^n\kappa_{l_i}. 
\end{align}
\end{theorem}

\begin{proof}
Consider the commutative diagram:
$$
\xymatrix{
\Moduli_{g,n} \ar[r]^{\rho_i} \ar[d]_{\pi_i}& \Moduli_{g,\{i\}} \ar[d]_{\Pi_i} \\
\Moduli_{g,n\smallsetminus\{i\}} \ar[r]^F & \Moduli_{g} \\
}
$$
For any Chow class $c$ in $\Moduli_{g, \{i\}}$ one can show, for instance by directly analyzing the definition of the four maps, or by appealing to the fact that $\Moduli_{g,n}$ maps birationally onto the fiber product $\Moduli_{g,n\smallsetminus \{i\}} \times_{\Moduli_{g,n}} \Moduli_{g,\{i\}}$, that
\begin{equation}
{\pi_i}_\ast \rho_i^\ast (c) =  F^\ast {\Pi_i}_\ast (c).
\end{equation}
Picking $c = \psi_i^{l+1}$, we have
\begin{equation}
{\pi_i}_\ast (\omega_i^{l+1}) =  F^\ast (\kappa_l).
\end{equation} 
For $i = 2, \ldots, n$, the class $\omega_i$ is pulled back via $\pi_1$. By projection formula:
\begin{equation}
{\pi_1}_\ast \left(\omega_1^{l_1+1} \prod_{i=2}^n \omega_i^{l_i+1} \right) =  F^\ast (\kappa_l) \prod_{i=2}^n \omega_i^{l_i+1}.
\end{equation} 
Pushing-forward via all other forgetful morphisms, and applying projection formula at each step, one obtains the result.
\end{proof}

Combining the results of Theorems \ref{num} and \ref{thm-ka}, one immediately obtains the following combinatorial formula relating $\kappa$ and $\psi$ top intersections.

\begin{corollary}
Let $g\geq 2$, and for $1\leq i \leq n$ let $l_i$ be a non-negative integer, with $ \sum_{i=1}^n l_i = 3g-3$. For any partition $\calP  = \{P_1,\dots, P_r \}
\vdash [n]$, define $\beta_j := \sum_{i\in P_j} l_i$.
Then:
\begin{align}
\label{omegaforka2}
 \int_{\Moduli_{g}}\prod_{i=1}^n\kappa_{l_i} =  \sum_{\calP \,\vdash [n]}(-1)^{n+\ell(\calP)} \int_{\Moduli_{g,\ell(\calP)}}\prod_{j=1}^{\ell(\calP)} \psi_{\bullet_j}^{\beta_j+1}.
\end{align}
\end{corollary}

\bibliographystyle{siam}
\bibliography{biblio}

\end{document}